\documentclass{amsart}
\usepackage{shortsalch, xy}
\xyoption{all}
\usepackage{enumitem}
\usepackage{geometry}
\newtheoremstyle{case}
{} 
{} 
{} 
{} 
{} 
{:} 
{ } 
{} 
\theoremstyle{case}
\newtheorem{case}{Case}
\newtheorem{subcase}{Case}
\numberwithin{subcase}{case}
\newtheoremstyle{steps}{}{}{}{}{}{.}{ }{}
\theoremstyle{steps}
\newtheorem{steps}{Step}
\def\co{\colon\thinspace}

\title{Maps of simplicial spectra whose realizations are cofibrations.}

\date{September 2017}
\author{G. Angelini-Knoll}
\author{A. Salch}

\begin{document}
\begin{abstract}
Given a map of simplicial topological spaces, mild conditions on degeneracies and the levelwise maps imply that the geometric realization of the simplicial map is a cofibration. These conditions are not formal consequences of model category theory, but depend on properties of spaces, and similar results have not been available for any model for the stable homotopy category of spectra. In this paper, we prove such results for symmetric spectra. Consequently, we get a set of conditions which ensure that the geometric realization of a map of simplicial symmetric spectra is a cofibration. These conditions are ``user-friendly'' in that they are simple, often easily checked, and do not require computation of latching objects. 
\end{abstract}
\maketitle
\tableofcontents
\section{Introduction.}

Let $f_{\bullet}\co X_{\bullet}\rightarrow Y_{\bullet}$ be a morphism of simplicial objects in some category of topological interest (e.g. spaces, spectra, equivariant spectra, $\dots$ ). It is often useful to know whether the induced map of geometric realizations
\begin{equation}\label{geom realiz map} \left| f_{\bullet}\right| \co \left| X_{\bullet}\right| \rightarrow \left| Y_{\bullet}\right|\end{equation}
is a cofibration. 

It is well-known that, given a model category $\mathcal{C}$, there exists a model structure (called the {\em Reedy model structure} due to its origin in C. Reedy's thesis~\cite{reedy}) on the category of simplicial objects in $\mathcal{C}$ such that, if $f_{\bullet}$ is a Reedy cofibration, then the map~\ref{geom realiz map} is a cofibration in the original model category $\mathcal{C}$. The map $f_{\bullet}$ is a {\em Reedy cofibration} if, for each nonnegative integer $n$, the latching comparison map
\[ X_n \coprod_{\tilde{L}_nX_{\bullet}} \tilde{L}_nY_{\bullet} \rightarrow Y_n\]
is a cofibration in $\mathcal{C}$. Here $\tilde{L}_n$ is a certain colimit called the {\em latching object} construction, whose definition we recall in Definition~\ref{latch}.

From the definition of a Reedy cofibration one can see that, in practical situations, it is often difficult to check that a given map $f_{\bullet}$ is a Reedy cofibration. The purpose of this paper is to give a straightforward, practical, often easily-checked set of conditions on a simplicial map of spectra $f_{\bullet}$ which ensure that it is a Reedy cofibration, and hence that the map~\ref{geom realiz map} is a cofibration of spectra. 

The main approach is to try to ``import" some classical results from unstable homotopy theory into the setting of symmetric spectra. In pointed topological spaces, it is typically quite easy to check that a simplicial map $f_{\bullet}\co X_{\bullet} \rightarrow Y_{\bullet}$ is a Reedy cofibration: 
\begin{enumerate}
\item{}\label{s1} if the degeneracy maps in $X_{\bullet}$ and $Y_{\bullet}$ are all closed cofibrations (i.e., $X_{\bullet}$ and $Y_{\bullet}$ are ``good'' simplicial spaces, in the sense of~\cite{MR0353298} ), then an easy application of Lillig's cofibration union theorem ~\cite{MR0334193} shows that $X_{\bullet}$ and $Y_{\bullet}$ are each Reedy cofibrant. 
\item{}\label{s2} Then one can show that, if $X_{\bullet}$ and $Y_{\bullet}$ are each Reedy cofibrant and each map $f_n \co X_n \rightarrow Y_n$ is a closed cofibration, then $f_{\bullet}$ is a Reedy cofibration.
\end{enumerate}
The key observation that makes this proof work is that the latching space $\tilde{L}_nX_{\bullet}$ is simply the union of the images of the degeneracy maps in $X_{\bullet}$, and so the natural map $\tilde{L}_nX_{\bullet}\rightarrow X_n$ is trivially seen to be a monomorphism, and using Lillig's theorem, a cofibration. As a consequence of this result, in the setting of simplicial topological spaces one only needs the degeneracy maps in $X_{\bullet}$ and $Y_{\bullet}$ to be closed cofibrations and for each map $f_n \co X_n \rightarrow Y_n$ to be a closed cofibration in order for the map~\ref{geom realiz map} to be a cofibration. See~\cite{MR0353298} and~\cite{MR0420610} for these facts from classical homotopy theory. 

Now one wants to be able to do something similar in stable homotopy theory, i.e., to replace spaces with (some model for) spectra. Problems immediately arise: for example, it is no longer necessarily true that the latching object $\tilde{L}_nX_{\bullet}$ is the ``union of the images of the degeneracy maps'' when $X_{\bullet}$ is a simplicial spectrum, so the classical proofs of steps~\eqref{s1} and~\eqref{s2}, above, do not ``work'' in spectra. If we are willing to work with the levelwise model structure on symmetric spectra in simplicial sets of topological spaces, then the problems that arise are easily surmountable (see Step~\ref{step1} of the proof of Theorem \ref{theorem1} and the proof of Theorem \ref{reedy levelwise cofibrancy lemma}), but the levelwise model structure does not have all the desired properties of a model for the homotopy category of spectra. It is therefore desirable to prove an analogue of the theorem for flat cofibrations (originally called $S$-cofibrations in Shipley \cite{MR2066511} and Hovey-Shipley-Smith \cite{MR1695653} where they were introduced). 

In this paper, we prove an analogue of this theorem for flat cofibrations in the setting of symmetric spectra in pointed simplicial sets, of~\cite{MR1695653}; see~\cite{schwedebook} for an excellent introduction and reference for symmetric spectra (though~\cite{schwedebook} is unpublished so we refer to published references instead whenever possible). 

As noted above, we make the distinction between ``levelwise cofibrations,'' ``flat cofibrations,''  ``positive levelwise cofibrations,'' and  ``positive flat cofibrations'' because the category of symmetric spectra has more than one useful notion of cofibration: see Definition~\ref{def of model structures} for a review of their definitions. The brief version is that levelwise cofibrations have a simple definition which is often easy to verify for a given map, while flat cofibrations give rise to a better-behaved model category (symmetric spectra with flat cofibrations admit a {\em symmetric monoidal} stable model structure \cite{MR1695653}, while with levelwise cofibrations the model structure fails to be monoidal \cite{schwedebook}), but the defining condition for a map to be a flat cofibration is significantly more difficult to verify.
Every flat cofibration is a levelwise cofibration, but the converse implication does not hold (for example $\bar{S}$ defined in Definition \ref{spectral latch} is not flat cofibrant because $\nu_2(\bar{S})$ is not a monomorphism with notation from Definition \ref{spectral latch}). The positive flat cofibrations require that one additional axiom be satisfied (see Definition~\ref{def of model structures}), but have the advantage that $E_{\infty}$-ring spectra are more easily modeled in the positive flat model structure on symmetric spectra (of topological spaces or simplicial sets); see~\cite{MR1695653} \cite{MR2066511} or~\cite{schwedebook} for more details.

If we combine Theorems~\ref{theorem1} and~\ref{reedy flat cofibration thm} together, we get a result which has some ``teeth'':
\begin{theorem}\label{bigthm}
Let $f_{\bullet}\co X_{\bullet} \rightarrow Y_{\bullet}$ be a map of simplicial symmetric spectra in $\sSet_*$. Suppose that all of the following conditions are satisfied:
\begin{itemize}
\item the symmetric spectra $X_n$ and $Y_n$ are flat-cofibrant for all $n$,
\item each of the degeneracy maps $s_i\co X_n \rightarrow X_{n+1}$ and $s_i\co Y_n \rightarrow Y_{n+1}$ are levelwise cofibrations,
\item and $f_n\co X_n \rightarrow Y_n$ is a flat cofibration for all $n$.
\end{itemize}
Then $f_{\bullet}$ is a Reedy flat cofibration. Consequently the map of symmetric spectra
\[ \left| f_{\bullet}\right| \co \left| X_{\bullet}\right| \rightarrow \left| Y_{\bullet}\right|\]
is a flat cofibration.

If we furthermore assume that each $X_n$ and each $Y_n$ is {\em positive} flat-cofibrant and that each $f_{n}$ is a {\em positive} flat cofibration, then
$f_{\bullet}$ is a Reedy positive flat cofibration, and consequently the
map of symmetric spectra
$\left| f_{\bullet}\right|$
is a positive flat cofibration.
\end{theorem}
We end the paper with an application of Theorem \ref{bigthm} to a practical problem. This problem arose for the authors in the process of constructing the ``THH-May spectral sequence'' of~\cite{thhmay}. 

It is a pleasure to thank E.~Riehl for helping us with our questions when we were trying to find out whether the results in this paper were already known.

\begin{conventions}\label{running conventions}
In this paper, we consider symmetric spectra in pointed simplicial sets and this category is denoted $\Sp_{\sSet_*}$. 
When referencing a model structure on $\sSet_*$, we will use the Quillen model structure throughout and refer to the cofibrations as simply as ``cofibrations." Recall that the cofibrations in this model structure are exactly the monomorphisms of pointed simplicial sets.  

This category of symmetric spectra has several important notions of cofibration, and four of them are used in this paper: the {\em levelwise cofibrations}, the {\em flat cofibrations}, the {\em positive levelwise cofibrations,} and the {\em positive flat cofibrations.} As a consequence, we have four ``Reedy'' notions of cofibration in the category of simplicial symmetric spectra. To keep them distinct, we will speak of ``Reedy levelwise cofibrations'' as opposed to ``Reedy flat cofibrations,'' and ``Reedy levelwise-cofibrant simplicial objects'' as opposed to ``Reedy flat-cofibrant simplicial objects,'' and so on.

When $X$ is a symmetric spectrum in $\mathcal{C}$ and $n$ is a nonnegative integer, we will write $X(n)$ for the level $n$ object of $X$. When $X_{\bullet}$ is a simplicial object, we will write $X_n$ for the $n$-simplices object of $X_{\bullet}$. So, for example, given a simplicial symmetric spectrum $X_{\bullet}$, we write $X(n)_{\bullet}$ for the simplicial object of $\mathcal{C}$ whose $m$-simplices object is $X(n)_m$. The symbols $X(n)_m$ and $X_m(n)$ have the same meaning, and we will use them interchangeably.

The word ``levelwise'' is used in two very different ways when speaking of maps between simplicial symmetric spectra, and for the sake of clarity, in this paper we consistently use the word ``pointwise'' instead of ``levelwise'' for one of these two notions. 
As an example, consider a definition from~\cite{schwedebook}: 
given a map of symmetric spectra $f\co X\rightarrow Y$, one says that $f$ is a {\em levelwise cofibration} if 
each of the component maps $f(n)\co X(n) \rightarrow Y(n)$ is a cofibration in $\sSet_*$. (See Definition~\ref{def of model structures}, below, for this definition and the related notion of a ``flat cofibration.'')
To distinguish this usage of the word ``levelwise'' from how the word ``levelwise'' is used when speaking of a map between simplicial objects, 
whenever we have a map $f_{\bullet}\co X_{\bullet} \rightarrow Y_{\bullet}$ of {\em simplicial} symmetric spectra,
we will say that $f_{\bullet}$ is a {\em pointwise flat cofibration}
if the map of symmetric spectra $f_n\co X_n \rightarrow Y_n$ is a flat cofibration for each $n$. 

When working with simplicial symmetric spectra, we will need to notationally distinguish between latching objects of simplicial objects, and latching objects of symmetric spectra; these notions are related but distinct (compare Definition~\ref{latch} and Definition \ref{spectral latch}). 
We will write $\tilde{L}_n(X_{\bullet})$ for the $n$th latching object of a simplicial object $X_{\bullet}$, and we write $L_n(X)$ for the $n$th latching object of a symmetric spectrum $X$ (see Definition \ref{spectral latch} or Construction~I.5.29 of~\cite{schwedebook} for this second notion).
\end{conventions}

\section{Review of the relevant model structures.}
The definition of the latching object of a simplicial object dates back to Reedy's thesis~\cite{reedy}, but there are a number of different (equivalent) versions of the construction.
The following version, which is convenient for what we do in this paper, appears as Remark~VII.1.8 in~\cite{MR1711612}. Also, see Proposition 15.2.6 in \cite{MR1944041} for a proof that this definition is equivalent to the other common definitions of latching objects in the Reedy model structure on functors from $\Delta^{\op}$ to a pointed model category $\mathcal{A}$. 
\begin{definition} \label{latch} Let $X_{\bullet}$ be a simplicial object in a pointed model category $\mathcal{A}$. By convention, let $\tilde{L}_0X_{\bullet}$ be the zero object in $\mathcal{A}$. The latching object $\tilde{L}_1X_{\bullet}$ is $X_0$ and the first latching map $\tilde{L}_1X_{\bullet}\rightarrow X_1$ is the degeneracy map $s_0\co X_0\rightarrow X_1$. 

For $n>1$ define \[ \tilde{L}_n X_{\bullet} := 
\text{coeq} 
\left \{ 
\xymatrix{ 
\coprod_{0\le i<j\le n-1} X_{n-2} \{i,j\}\ar@<1ex>[r]^(.55){S^{\prime}} \ar@<-1ex>[r]_(.55){S^{\prime \prime}} & \coprod_{k=0}^{n-1} X_{n-1} \{k\}  
} \right \} \] 
with $S^{\prime}$ and $S^{\prime \prime}$ defined as follows:
for a given pair $(i,j)$ with $i<j$ we define maps
\[ \xymatrix{ X_{n-2} \{i,j\} \ar[r]^{s_i}& X_{n-1}\{j\} \ar[r]^(.4){\iota_j} & \coprod_{k=0}^{n-1} X_{n-2}\{k\}}\] 
and 
\[ \xymatrix{ X_{n-2} \{i,j\} \ar[r]^{s_{j-1}}& X_{n-1}\{i\}\ar[r]^(.4){\iota_i} & \coprod_{k=0}^{n-1} X_{n-2}\{k\}}\]
where $s_i,s_{j-1}$ are the degeneracy maps in our simplicial object and $\iota_k$ is the inclusion into the $k$-th summand. We then define $S^{\prime}$ using the first collection of maps and the universal property of the coproduct and we define $S^{\prime\prime}$ using the second collection of maps and the universal property of the coproduct. The symbols $\{j\}$ and $\{i,j\}$ are simply formal symbols used to index the coproduct summands. We have a map $\coprod_{k=0}^{n-1} X_{n-1} \{k\} \lra X_n$ given by the coproduct of the degeneracies and this produces a natural comparison map 
\[ \tilde{\nu}_n(X_{\bullet})\co \tilde{L}_nX_{\bullet}\lra X_n \] 
by universal property of the coequalizer and the simplicial identity $s_js_i=s_is_{j+1}$. (See~\cite{MR1711612} for more details on the natural transformation $\tilde{\nu}_n(-)$.) 
\end{definition}

\begin{theorem}[ Theorem 15.3.4 in \cite{MR1944041}] There is a model structure on simplicial objects in a model category $\mathcal{A}$ called the \emph{Reedy model structure}, where the cofibrations are the maps $X_{\bullet}\rightarrow Y_{\bullet}$ such that, for each $n$, the map
\[ X_n \coprod_{\tilde{L}_nX_{\bullet}} \tilde{L}_nY_{\bullet} \rightarrow Y_n\]
is a cofibration in $\mathcal{A}$, the weak equivalences are the pointwise weak equivalences, and the fibrations are the Reedy fibrations \cite[Definition 15.3.3]{MR1944041}.
\end{theorem}
\begin{remark} We do not define the fibrations in the Reedy model structure because they are not used in our paper, but they can be defined in a similar (dual) way to the cofibrations using the matching object construction. \end{remark}
\begin{example}
The main examples of Reedy model categories that we will be interested in will be the Reedy model structure on simplicial objects in symmetric spectra in pointed simplicial sets (for various model structures on symmetric spectra in pointed simplicial sets) and the Reedy model structure on pointed bisimiplicial sets (where we think of them as simplicial objects in the category of pointed simplicial sets with the Quillen model structure). 
\end{example}
The following definition appears as Definition 5.2.1 in Hovey-Shipley-Smith \cite{MR1695653}. 
\begin{definition}\label{spectral latch}
Define $\bar{S}$ to be the symmetric spectrum with $\bar{S}_n=S^n$ for $n\ge 1$ and $\bar{S}_0=*$ with the evident structure maps. Given a symmetric spectrum $X$ define $L_nX$ to be $(\bar{S}\smash X)_n$. The evident map $\bar{S}\rightarrow S$ of symmetric spectra produces a natural transformation $\nu_n(-)\co(\bar{S}\smash -)_n\rightarrow (S\smash -)_n$. 
\end{definition}
\begin{definition}\label{def of model structures}
Let $\Sp_{\sSet_*}$ denote the category of symmetric spectra in pointed simplicial sets.  A map $f\co X\rightarrow Y$ in $\sSet_*$ is said to be:
\begin{itemize}[leftmargin=*]
\item a {\em levelwise cofibration} if, for all nonnegative integers $n$, the map $f(n) \co X(n) \rightarrow Y(n)$ is a cofibration in $\sSet_*$, 
\item a {\em positive levelwise cofibration} if $f$ is a levelwise cofibration and $f(0)\co X(0) \rightarrow Y(0)$ is an isomorphism,
\item  a {\em flat cofibration} if, for all nonnegative integers $n$, the latching map $X(n) \underset{L_n X}{\coprod} L_nY \rightarrow Y(n)$ is a cofibration, 
\item and a {\em positive flat cofibration} if $f$ is a flat cofibration and $f(0)\co X(0) \rightarrow Y(0)$ is an isomorphism.
\end{itemize}
\end{definition}
\begin{remark}
The flat cofibrations were first defined in  Hovey-Shipley-Smith \cite{MR1695653} and Shipley \cite{MR2066511} where they are called ``$S$-cofibrations." Following Schwede, we refer to these cofibrations as the flat cofibrations. Note that the definition of the flat cofibration in a more general simplicial pointed model category $\mathcal{C}$ requires that the map 
\[X(n) \coprod_{L_n X} L_nY \rightarrow Y(n) \]
 is a $\Sigma_n$-cofibration.  (See Definition 3.6 in \cite{2013arXiv1307.4488G} for the definition of $G$-cofibrations in a model category $\mathcal{C}$ enriched in a cosmos $\mathcal{V}$ where $G$ is a finite group. In \cite{2013arXiv1307.4488G}, they define more general $\mathcal{F}$-cofibrations for a family of subgroups $\mathcal{F}$ of $G$ and the $G$-cofibrations correspond to taking the family of all subgroups.) Since pointed simplicial sets have the property that $G$-cofibrations are equivalent to the monomorphisms of pointed simplicial sets, or in other words the cofibrations after forgetting the group action, the definition of flat cofibration in Definition \ref{def of model structures} is equivalent to the definition in a general simplicial pointed model category, see Proposition~2.16 of~\cite{Stephan} for a proof of this fact. 
 \end{remark}
The content of Theorem~\ref{existence of flat model structure} below is contained in Theorem III.4.11 of ~\cite{schwedebook}. For a published account in the setting of simplicial sets and topological spaces see Corollary 5.3.8 in Hovey-Shipley-Smith \cite{MR1695653} as well as Proposition 2.2, Theorem 2.4, and Proposition 2.5 in \cite{MR2066511}. To make sure that $\sSet_*$ satisfies the assumptions of Theorem III.4.11 in ~\cite{schwedebook}, we apply Proposition 1.3 in Shipley \cite{MR2066511}, which states that $\sSet_*$ can be equipped with the mixed equivariant model structure. 
\begin{theorem}\label{existence of flat model structure}
Let $\sSet_*$ be the category of pointed simplicial sets. 
Then the flat cofibrations are the cofibrations of a model structure on $\Sp_{\sSet_*}$ called the {\em absolute flat stable model structure.} 
while the positive flat cofibrations are the cofibrations of a model structure on $\Sp_{\sSet_*}$ called the {\em positive flat stable model structure,}  

Equipped with either of these model structures, $\Sp_{\sSet_*}$ is a symmetric monoidal stable model category satisfying the pushout-product axiom and the homotopy category $\Ho(\Sp_{\sSet_*})$
is equivalent to the classical stable homotopy category, with its smash product.
\end{theorem}

The absolute flat model structure and the positive flat model structure have equivalent homotopy categories, and have many good properties properties in common. The absolute flat model structure has the advantage of being slightly simpler, while the positive flat model structure has some better properties than the absolute flat model structure when one wants to work with structured symmetric ring spectra. See Section 5 in Hovey-Shipley-Smith \cite{MR1695653}  and \cite{schwedebook} for a discussion of some of the nice properties of the flat model structures. The category of symmetric spectra does {\em not} have a well-behaved symmetric monoidal product when equipped with the levelwise cofibrations or the positive levelwise cofibrations. See~\cite{schwedebook} 
for details. 

Note that the cofibrations in the absolute flat stable model structure agree with those in the flat level model structure. Since we are only working with cofibrations in this paper the results hold in any model structure on symmetric spectra in pointed simplicial sets where the cofibrations are exactly the flat cofibrations. Similarly, the results about positive flat cofibrations hold in any model structure on symmetric spectra in pointed simplicial sets where the cofibrations are exactly the positive flat cofibrations. 
\section{Reedy cofibrant objects: sufficient conditions.}
In classical references on simplicial spaces (see e.g.~\cite{MR0353298}), a simplicial topological space is called ``good" if its degeneracy maps are all closed cofibrations; we give an analogous definition of a ``good" simplicial spectrum. 
\begin{definition} \label{good} 
By a \emph{good simplicial symmetric spectrum in $\sSet_*$},
we mean a simplicial object $X_{\bullet}$ in the category $\Sp_{\sSet_*}$ such that 
\begin{enumerate} 
\item{} $X_{n}$ is a flat-cofibrant symmetric spectrum for each $n$, and
\item{} the degeneracy maps $s_i\co X_{n} \lra X_{n+1}$ are levelwise cofibrations for each $n$ and $i$. 
\end{enumerate}

We will say that $X_{\bullet}$ is {\em positive-good} if $X_{\bullet}$ is good
and
$X_{n}$ is positive flat-cofibrant for each $n$.
\end{definition}
Definition~\ref{good} is unusual-looking, because it refers to two different model structures (``flat'' and ``levelwise''). Here is some explanation:
every flat cofibration is also a levelwise cofibration, so if $X_{\bullet}$ is a simplicial symmetric spectrum which is pointwise flat-cofibrant and whose degeneracies are all flat cofibrations, then $X_{\bullet}$ is good. In Definition~\ref{good}, however, we only ask for the degeneracies to be levelwise cofibrations, not flat cofibrations, because:
\begin{enumerate}
\item checking that a map is a levelwise cofibration in a specific case of interest is typically much easier than checking that it is a flat cofibration, and
\item our main theorem in this section, Theorem~\ref{theorem1}, only needs the degeneracy maps to be levelwise cofibrations, not necessarily flat cofibrations.
\end{enumerate}

Finally, we state the main theorem in this section of the paper, which 
provides an easy-to-check criterion for a simplicial symmetric spectrum to be Reedy flat-cofibrant. The rest of this section is devoted to proving this theorem. 

\begin{theorem} \label{theorem1} 
Let $X_{\bullet}$ be a good simplicial symmetric spectrum in $\sSet_*$. 
Then $X_{\bullet}$ is Reedy flat-cofibrant; i.e. the map $\tilde{L}_nX_{\bullet} \lra X_{n}$ is a flat cofibration for each $n$. If $X_{\bullet}$ is furthermore assumed to be positive-good, then $X_{\bullet}$ is Reedy positive flat-cofibrant. 
\end{theorem}
\begin{proof} We will prove this theorem in multiple steps. 
\begin{steps}\label{step1}
First, we will show that any good simplicial symmetric spectrum in $\sSet_*$ is Reedy levelwise cofibrant; i.e. 
the map $\tilde{L}_nX_{\bullet} \lra X_{n}$ of symmetric spectra is a levelwise cofibration for each $n$. We first consider, for each fixed nonnegative integer $n$, the pointed bisimplicial set $X(n)_{\bullet}$ whose pointed simplicial set of $m$ simplices is obtained by taking the level $n$ pointed simplicial set of the symmetric spectrum of $m$-simplices in the simplicial symmetric spectrum $X_{\bullet}$. By Corollary 15.8.8 of \cite{MR1944041}, it is easy to deduce that all bisimplicial sets are Reedy cofibrant. Since the model strucure on pointed bisimplicial sets is created by the forgetful functor, the same statement is true for pointed bisimplicial sets. (Note that all three natural Reedy model structures on $((\Sets_*)^{\Delta^{\op}})^{\Delta^{\op}}$ are equivalent as model categories by Theorem 15.5.2 of \cite{MR1944041}.) 
Consequently, the map 
\[ \tilde{L}_m(X(n)_{\bullet})\rightarrow X(n)_m \]
is a cofibration of pointed simplicial sets. 
Since the evaluation functor which takes a simplicial symmetric spectrum $X_{\bullet}$ to a pointed bisimplicial set $X(n)_{\bullet}$
is a left adjoint (see 
Example 4.2 \cite{schwedebook}) 
and the simplicial latching construction $\tilde{L}_m(X_{\bullet})$ is constructed as a colimit by Definition \ref{latch} the two commute. Hence, the map 
$(\tilde{L}_mX_{\bullet})(n)\rightarrow X_m(n) $
is a cofibration for each $n$, in other words 
$ \tilde{L}_mX_{\bullet}\rightarrow X_m $
is a levelwise cofibration. \\
\end{steps}
\begin{steps}\label{step2}
We will use the following observation: suppose we have a diagram in pointed simplicial sets of the form 
\[\xymatrix{
PB \ar[d]^{g_1}  \ar[r]^{g_2}  & B \ar[d]_{h_1} \ar[ddr]^{f_2} & \\
C  \ar[r]^{h_2}  \ar[drr]_{f_1} & P \ar[dr]^(.3){F} &  \\
&& D  }\]
where $PB$ is the pullback $C\prod_DB$, $P$ is the pushout $C\coprod_{PB} B$, and $f_1$ and $f_2$ are monomorphisms of pointed simplicial sets. Then we want to show that $F\co P\lra D$ is a monomorphism. 
To prove this claim, we observe that colimits and limits are computed pointwise in pointed simplicial sets because it is a category of functors (see e.g.~\cite{MR1300636}). The proof for pointed sets is an easy exercise and is therefore left to the reader. \\

In order to prove that $X_{\bullet}\in \text{ob}\Sp_{\sSet_*}^{\Delta^{\text{op}}}$ is Reedy flat-cofibrant and not just Reedy levelwise-cofibrant, we need to know that the map $F$ in the commutative diagram 
\begin{equation}\label{diag latch}
 \xymatrix{
L_s \tilde{L}_nX_{\bullet} \ar[r] \ar[d]  & L_s (X_n )  \ar[d] \ar[ddr]  &  \\ 
(\tilde{L}_nX_{\bullet})(s) \ar[r]   \ar[rrd] &           PO_1  \ar[dr]^(.3){F}  &  \\
&&   (X_{n})(s) } \end{equation}
is a cofibration in $\sSet_*$, where $PO_1$ is defined as the pushout $(\tilde{L}_nX_{\bullet})(s) \coprod_{L_s \tilde{L}_nX_{\bullet} } L_s (X_n)$. We break this into two further steps \\
\end{steps}
\begin{steps}\label{step3}
Define $PB$ to be the pullback $(\tilde{L}_nX_{\bullet})(s)\prod_{(X_{n})(s)}L_s (X_n )$ in $\sSet_*$. 
Then we will show that the universal map $F'$ in the commutative diagram 
\begin{equation} \label{eq 2 cor}
\xymatrix{
PB \ar[r]^{g_2} \ar[d]_{g_1}  & L_s (X_n )  \ar[d] \ar[ddr]^{f_1}  &  \\ 
(\tilde{L}_nX_{\bullet})(s) \ar[r]   \ar[rrd]_{f_2} &           PO_2  \ar[dr]^(.3){F'}  &  \\
&&   (X_{n})(s) } \end{equation} 
in $\sSet_*$ 
 is a cofibration, where $PO_2$ is the pushout $(\tilde{L}_nX_{\bullet})(s)\coprod_{PB}L_s (X_n )$.

The map $f_1\co L_s(X_n) \rightarrow (X_n)(s)$ in diagram~\ref{eq 2 cor}
is a cofibration in $\sSet_*$ since, from the definition of a good simplicial spectrum, the spectrum $X_n$ is flat-cofibrant for each $n$, the map $f_2$ in diagram~\ref{eq 2 cor} is a cofibration in $\sSet_*$ by Step~\ref{step1}. Since cofibrations in $\sSet_*$ are exactly the monomorphisms, $f_1$ and $f_2$ are monomorphisms of pointed simplicial sets and the fact that $F'$ is a cofibration
follows by Step~\ref{step2}.
 
Note that by universal properties, we have maps $G\co L_s \tilde{L}_nX_{\bullet} \lra PB$ and $F^{\prime\prime}\co PO_1\lra PO_2$
and each of these maps fits into Diagram~\ref{diag}, below.
Since $F^{\prime}$ is a cofibration and $F=F^{\prime}\circ F^{\prime\prime}$, we just need to show that $F^{\prime\prime}$ is a cofibration in $\sSet_*$ and that will imply $F$ is a cofibration in $\sSet_*$. 
We claim that we just need to prove that the map $G_m$ is a surjection for each $m$, where $G_m\co(L_s \tilde{L}_nX_{\bullet})_m\rightarrow PB_m$ is the map of pointed sets induced by the evaluation functor $(-)_m\co\sSet_*\rightarrow \Sets_*$. The claim follows by the following argument: 
given the commutative diagram of pointed sets induced by applying the evaluation functor $(-)_m$ to the diagram 
\begin{equation}\label{diag} \xymatrix{ 
L_s\tilde{L}_nX_{\bullet} \ar[dr]^(.7){G} \ar[ddr]_{g_1} \ar[drr]^{g_2}  &&& \\
&PB \ar[r]_{\ell_1} \ar[d]^{\ell_2}  & L_s (X_n ) \ar[dddrr]^{h_1} \ar[d]_{f_1} \ar[ddr]  & &  \\ 
&(\tilde{L}_nX_{\bullet})(s) \ar[r]^{f_2} \ar[drr] \ar[ddrrr]_{h_2}   \ar[rrd] &  PO_1  \ar[dr]^(.3){F^{\prime\prime}} & & \\
&&& PO_2 \ar[dr]^(.2){F^{\prime}} & \\
&& &&  (X_{n})(s), } \end{equation}
 with monomorphisms of pointed sets $(\tilde{L}_nX_{\bullet}(s))_m\rightarrow (X_n(s))_m$, $(L_s(X_n))_m\rightarrow (X_n(s))_m$, and $(PO_2)_m\rightarrow  ((X_{n})(s))_m$,
 and the identifications
 \[ (PO_1)_m=  \left(\tilde{L}_nX_{\bullet}\right)(s))_{m}\underset{(L_s\tilde{L}_nX_{\bullet})_m}{\coprod} (L_s(X_n))_m,\] 
 \[ (PO_2)_m= \left(\tilde{L}_nX_{\bullet}(s)\right)_m  \underset{(PB)_m}{\coprod}(L_s(X_n))_m,\text{ and }\]
 \[(PB)_m=\left( \tilde{L}_nX_{\bullet}(s)\right)_m \underset{\left(X_n(s)\right)_m}{\prod}(L_s(X_n))_m,\] 
 (which hold because $(-)_m$ is both a right and a left adjoint and therefore commutes with limits and colimits in pointed simplicial sets), then if in addition 
\[\left(L_s\tilde{L}_nX_{\bullet}\right)_m \rightarrow (PB)_m  \]
is an epimorphism, then it is easy to check that $(PO_1)_m \rightarrow (PO_2)_m$ is a monomorphism. Hence, it suffices to show that $G_m$ is surjective for all $m$ in order to show that $F^{\prime\prime}\co PO_1\rightarrow PO_2$ is a cofibration of pointed simplicial sets. 
\end{steps}
\begin{steps}\label{step4}
By the previous step, given the commutative diagram \eqref{diag}, it suffices to show that the 
map of pointed sets $G_m\co \left(L_s\tilde{L}_nX_{\bullet}\right)_m \lra PB_m$ is an epimorphism for each $m$ in order to show that $F^{\prime\prime}$ is a cofibration. Throughout, we use the fact that colimits and limits in pointed simplicial sets are computed pointwise and colimits and limits in symmetric spectra in pointed simplicial sets are computed levelwise. 

Let $\bar{z}\in PB_m$. Then since $PB_m$ is a pullback, $\bar{z}$ is represented by elements
\[ x_1=(\ell_1)_m(\bar{z})\in (L_s X_n)_m \text{ and }x_2=(\ell_2)_m(\bar{z})\in \left(\tilde{L}_nX_{\bullet})(s)\right)_m \] 
such that $(h_1)_m(x_1)=(h_2)_m(x_2)$.
Since $x_2\in \left(\tilde{L}_nX_{\bullet}(s)\right)_m = \coprod_{k=0}^{n-1} (X_{n-1}(s))_m\{k \}/\sim$, it can be chosen as an equivalence class of some element in $(X_{n-1}(s)))_m\{j\}$ for some $j$. 
Every element in $X_{n-1}(s)$ is a face of some element in $X_n(s)$, so we can choose $j$ so that the composite 
\[ \xymatrix{ X_n\ar[r]^{d_j} & X_{n-1}\{j\} \ar[r] & \coprod_{k=0}^{n-1} X_{n-1}\{k\}\ar[r]& \coprod_{k=0}^{n-1} X_{n-1}\{k\}/\sim , }\]
which we call $\bar{d}_j$, satisfies $((\bar{d})j(s)_m\circ (h_2)_m)(x_2)=x_2$.

By functoriality of the evaluation functor $(-)_m$ and the spectral latching functor $L_s$ there is also a map $\left(L_s(\bar{d}_j)\right)_m\co  (L_s X_n)_m \lra  (L_s \tilde{L}_nX_{\bullet})_m$. We claim the following: 
\begin{enumerate}
\item{} \label{i1} $(g_1\circ L_s(\bar{d}_j)))_m(x_1)=x_2,$  and
\item{} \label{i2} $(g_2\circ L_s(\bar{d}_j))_m(x_1)=x_1.$ 
\end{enumerate} 
Item~\eqref{i1} follows by naturality of $\nu_s$, which we explain as follows. First, we know $(g_1)_m=(\nu_s(\tilde{L}_nX_{\bullet}))_m$ and naturality states that the diagram 
\[ 
\xymatrix{ 
  (L_sX_n)_m \ar[r]^{(L_s(\bar{d}_j))_m} \ar[d]_{(\nu_s(X_{n}))_m}&   (L_s\tilde{L}_nX_{\bullet})_m \ar[d]^{(\nu_s(\tilde{L}_nX_{\bullet}))_m}\\
(X_n(s))_m  \ar[r]^{(\bar{d}_j (s))_m}&   (\tilde{L}_n(X_{\bullet}(s)))_m
}
\]
commutes; i.e., 
\[ (\nu_s(\tilde{L}_nX_{\bullet})\circ (L_s(\bar{d}_j)))_m(x_1) = (\bar{d}_j(s)\circ \nu_s(X_n))_m(x_1)\] 
We then use the fact that $h_1=\nu_s(X_n)$ and the formula $(h_1)_m(x_1)=(h_2)_m(x_2)$ to produce
\[ (\bar{d}_j(s))_m((h_1)_m(x_1))=(\bar{d}_j(s))_m((h_2)_m(x_2)). \]
This combines with the fact that $(\bar{d}_j\circ h_2)_m(x_2))=x_2$ to produce
\[ (g_1\circ L_s(\bar{d}_j))_m(x_1)=x_2 \] 
as desired. 

To prove Item~\eqref{i2}, note that by naturality of  $\nu_s$ the diagram 
\[ \xymatrix{
(L_s(X_n))_m \ar[r]^{(L_s( \bar{d}_j))_m}  \ar[d]_{(\nu_s(X_n))_m }& (L_s\tilde{L}_nX_{\bullet})_m \ar[r]^{(L_s(\tilde{\nu}(X_{\bullet})))_m} \ar[d]^{(\nu_s(\tilde{L}_n(X_{\bullet})))_m} & (L_s(X_n))_m \ar[d]^{(\nu_s(X_n))_m}  \\
(X_n(s))_m \ar[r]^{(\bar{d}_j(s))_m}  & (\tilde{L}_nX_{\bullet}(s))_m \ar[r]^{(\tilde{\nu}(X_{\bullet})(s))_m} & (X_n(s))_m   }\]
commutes. We know that $h_1=\nu_s(X_n)$, so
\begin{equation}\label{formula1} (h_1)_m\circ (L_s(\tilde{\nu}(X_{\bullet})))_m\circ (L_s(\bar{d}_j))_m(x_1)=(\tilde{\nu}(X_{\bullet})(s))_m\circ (\bar{d}_j(s))_m \circ (h_1)_m(x_1) \end{equation}
and since $(h_1)_m(x_1)=(h_2)_m(x_2)$ and $(\bar{d}_j(s)\circ h_2)_m(x_2)=x_2$, we know that
\[ (\bar{d}_j(s) \circ h_1)_m(x_1)=(\bar{d}_j(s)\circ h_2)_m(x_2))=x_2, \]
and hence that  
\[ (\tilde{\nu}(X_{\bullet})(s)\circ \bar{d}_j(s) \circ h_1)_m(x_1) = (\tilde{\nu}(X_{\bullet})(s) \circ \bar{d}_j(s) \circ h_2)_m(x_2) =(\tilde{\nu}(X_{\bullet}(s)))_m(x_2). \] 
Now note that $\tilde{\nu}(X_{\bullet})(s)=h_2$ so 
\[(\tilde{\nu}(X_{\bullet}(s)))_m(x_2)=(h_2)_m(x_2)=(h_1)_m(x_1) \] 
and hence, by Equation \eqref{formula1}, 
\[ (h_1)_m\circ (L_s(\tilde{\nu}(X_{\bullet})))_m\circ (L_s(\bar{d}_j))_m(x_1))= (h_1)_m(x_1)\]
Since $X_{\bullet}$ is pointwise flat-cofibrant, $(h_1)_m$ is a monomorphism, so it is left cancellable and therefore 
\[ ((L_s (\tilde{\nu}(X_{\bullet}) )\circ L_s(\bar{d}_j))_m(x_1)=x_1.\] 
Now note that $(g_2)_m=(L_s (\tilde{\nu}(X_{\bullet})))_m$ by definition, so we have proven the claim. 

Thus, given an element $\bar{z}$ in the pullback represented by $x_1$ and $x_2$, we have constructed an element $z=(L_s(\bar{d}_j))_m(x_1))\in (L_s(\tilde{L}_nX_{\bullet}))_m$, such that $(g_1)_m(z)=x_2$ and $(g_2)_m(z)=x_1$  and hence $G_m(z)=\bar{z}$ as desired. So $G_m$ is surjective for all $m$. 
\end{steps}
\begin{steps}\label{step5}
The goal was to prove that a good symmetric spectrum in the flat model structure is Reedy flat-cofibrant. By Step~\ref{step1}, we know that a good symmetric spectrum in the flat model structure is Reedy levelwise-cofibrant, so we just need to elevate the map $\tilde{L}_nX_{\bullet}\lra X_n$ to a flat cofibration. 
It suffices to show that the map $F$ in the commutative diagram \eqref{diag latch} is a cofibration in $\sSet_*$. We then write $F$ as a composite $F^{\prime}\circ F^{\prime\prime}$. The map $F^{\prime}$ is a cofibration in $\sSet_*$ by Step~\ref{step2}, and the map $F^{\prime\prime}$ is a cofibration in $\sSet_*$ by Step~\ref{step4} and Step~\ref{step5}. Hence, $F$ is a cofibration. 

Now suppose furthermore that $X_{\bullet}$ is positive-good. We need to know that
$X_{\bullet}$ is also cofibrant in the Reedy positive-flat model structure, i.e., that $\nu_n(X_{\bullet})\co \tilde{L}_nX_{\bullet} \rightarrow X_n$ is a positive flat cofibration. Since we have already shown that $\nu_n$ is a flat cofibration, all that remains is to show that $\nu_n(X_{\bullet})(0)\co (\tilde{L}_nX_{\bullet})(0) \rightarrow X_n(0)$ is an isomorphism, i.e., that $\nu_n(X(0)_{\bullet})\co \tilde{L}_n(X(0)_{\bullet}) \rightarrow X(0)_n$ is an isomorphism.
Since $X_{\bullet}$ is positive-good, each $X_n$ is positive flat-cofibrant, so $X(0)_n \cong 0$ for all $n$. This implies that $X(0)_{\bullet}$ is the zero object in pointed simplicial sets, so its latching objects are also all the zero object in pointed sets, hence, $\nu_n(X(0)_{\bullet})$ is an isomorphism, as desired.
\end{steps}
\end{proof} 

\section{Reedy cofibrations: sufficient conditions.}
The proof of Theorem \ref{bigthm} easily follows by combining Theorem \ref{theorem1} with Theorem \ref{reedy levelwise cofibrancy lemma} and Theorem \ref{reedy flat cofibration thm}, which we prove in this section. 
\begin{theorem} \label{reedy levelwise cofibrancy lemma}
Let $X_{\bullet}\stackrel{f_{\bullet}}{\longrightarrow} Y_{\bullet}$ be a morphism of simplicial symmetric spectra in $\sSet_*$. Make the following assumptions:
\begin{itemize}
\item The map $f_{\bullet}$ is a pointwise flat cofibration. That is, for each nonnegative integer $n$, the map of spectra $X_n \stackrel{f_n}{\longrightarrow} Y_n$ is a flat cofibration.
\item 
The simplicial spectrum $Y_{\bullet}$ is Reedy levelwise-cofibrant.
\end{itemize}

Then $f_{\bullet}$ is a Reedy levelwise cofibration. Hence, $|f_{\bullet}|$ is a levelwise cofibration.

If we furthermore assume that $f_{\bullet}$ is a pointwise {\em positive} flat cofibration, then $f_{\bullet}$ is also a Reedy positive levelwise cofibration. Hence $|f_{\bullet}|$ is a positive levelwise cofibration. 
\end{theorem}
\begin{proof}
For each nonnegative integer $n$, we need to show that the map of symmetric spectra
\[ X_{n} \coprod_{\tilde{L}_n(X_{\bullet})}\tilde{L}_n(Y_{\bullet}) \rightarrow Y_{n}\]
is a levelwise cofibration, i.e., that 
\begin{equation}\label{map 4398} \left(X_{n} \coprod_{\tilde{L}_n(X_{\bullet})}\tilde{L}_n(Y_{\bullet})\right)(m) \rightarrow Y_{n}(m)\end{equation}
 is a cofibration in $\sSet_*$ for all nonnegative integers $m,n$.
Now colimits in $\Sp_{\sSet_*}$ are computed levelwise (see example~I.3.5 in~\cite{schwedebook}), 
so map~\ref{map 4398} agrees, up to an isomorphism, with the map
\begin{equation}\label{map 4399} c_{n}(m)\co X_{n}(m) \coprod_{\left(\tilde{L}_n(X_{\bullet})\right)(m)}\left(\tilde{L}_n(Y_{\bullet})\right)(m) \rightarrow Y_{n}(m).\end{equation}

We again write $(-)_\ell\co\sSet_*\rightarrow \Sets_*$ for the usual evaluation functor at a nonnegative integer $\ell$. As in all functor categories, colimits in $\sSet_*$ are computed pointwise and the monomorphisms in $\sSet_*$, which are also the cofibrations in $\sSet_*$, are the pointwise monomorphisms. Hence, if we can show that $(c_{n}(m))_{\ell}$ is a monomorphism for each nonnegative integer $\ell$, then $c_{n}(m)$ is a cofibration, and we are done. Since $(-)_\ell$ preserves finite colimits, the map $(c_{n}(m))_\ell$ agrees (up to an isomorphism in the domain) with the map
\[ (\tilde{c}_{n}(m))_\ell\co \left(X_{n}(m)\right)_\ell \coprod_{\left(\left(\tilde{L}_n(X_{\bullet})\right)(m)\right)_\ell}\left(\left(\tilde{L}_n(Y_{\bullet})\right)(m)\right)_\ell \rightarrow \left(Y_{n}(m)\right)_\ell.\]

Now suppose that 
\[ x_0,x_1\in \left((X_{n})(m)\right)_\ell \coprod_{\left(\left(\tilde{L}_n(X_{\bullet})\right)(m)\right)_\ell}\left(\left(\tilde{L}_n(Y_{\bullet})\right)(m)\right)_\ell \]
satisfy $(\tilde{c}_{n}(m))_\ell(x_0) = (\tilde{c}_{n}(m))_\ell(x_1)$.
Then, by the usual description of pushouts in the category of pointed sets as unions, there are three possibilities: 

\begin{case} Suppose $x_0$ and $x_1$ are both in $\left((X_{n})(m)\right)_\ell$. 
Since $X_n \rightarrow Y_n$ is a flat cofibration, it is also a levelwise cofibration (see~\cite{schwedebook}), 
i.e.,
the map $X_n(m) \rightarrow Y_n(m)$ is a cofibration in $\mathcal{C}$ for all $n$. Hence 
$(X_n(m))_\ell \rightarrow (Y_n(m))_\ell$ is a monomorphism
and hence $x_0 = x_1$.
\end{case}

\begin{case}
Suppose $x_0$ and $x_1$ are both in $\left(\left(\tilde{L}_n(Y_{\bullet})\right)(m)\right)_\ell$. 
Since $Y_{\bullet}$ is Reedy levelwise-cofibrant, the map
$\tilde{L}_n(Y_{\bullet}) \rightarrow Y_n$ is a levelwise cofibration, i.e.,
$(\tilde{L}_n(Y_{\bullet}))(m) \rightarrow Y_n(m)$
is a cofibration in $\sSet_*$ for all $m$, and consequently
$(\tilde{L}_n(Y_{\bullet})(m))_\ell \rightarrow (Y_n(m))_\ell$
is a monomorphism. 
Hence $x_0 = x_1$. 
\end{case}

\begin{case}
Suppose $x_0$ is in $\left(X_{n}(m)\right)_\ell$ and $x_1$ is in $\left(\left(\tilde{L}_n(Y_{\bullet})\right)(m)\right)_\ell$. 
The same argument as below also works in the case $x_1$ is in $\left(X_{n}(m)\right)_\ell$ and $x_0$ is in $\left(\left(\tilde{L}_n(Y_{\bullet})\right)(m)\right)_\ell$.
This part requires a bit more thought than the previous parts. 
Given any finitely complete, finitely co-complete category $\mathcal{A}$ and any simplicial object $Z_{\bullet}$ of $\mathcal{A}$, the latching object $\tilde{L}_n(Z_{\bullet})$ of $Z_{\bullet}$ is isomorphic, by Definition~\ref{latch}, to the coequalizer of a pair of maps whose codomain is a coproduct of $n$ copies of $Z_{n-1}$, namely, one for each degeneracy map $Z_{n-1}\rightarrow Z_{n}$, and the domains as well as the maps themselves are built from finite limits and finite colimits of copies of $Z_m$ for various $m<n-1$ and the degeneracy maps connecting them. If $\mathcal{A}$ is the category of sets (or pointed sets), then $\tilde{L}_n(Z_{\bullet})$ is simply a coproduct of $n$ copies of $Z_{n-1}$ modulo equivalence relations coming from identifying subsets of the copies of $Z_{n-1}$ given by intersections of copies of $Z_m$ for $m<n-1$. For each $k\in \{ 0, \dots ,n-1\}$ we have a map $\overline{d}_k\co Z_n \rightarrow \tilde{L}_nZ_{\bullet}$ given by applying the face map $d_k\co Z_n \rightarrow Z_{n-1}$ and then including $Z_{n-1}$ as the $k$th coproduct summand in $\tilde{L}_nZ_{\bullet} = \left(\coprod_{i=0}^{n-1} Z_{n-1}\right)/\sim.$

Now since taking the $m$-th pointed simplicial set is a functor from $\Sp_{\sSet_*}$ to $\sSet_*$, applying the functor $Z \mapsto Z(m)$ to a simplicial symmetric spectrum yields a simplicial object of $\sSet_*$. 
As limits and colimits in symmetric spectra are computed levelwise (see Example~I.3.5 in~\cite{schwedebook}), 
this functor $Z\mapsto Z(m)$ also commutes with limits and colimits.
Let
$X_{\bullet}(m), Y_{\bullet}(m)$ denote the bisimplicial pointed set obtained by applying the $m$th ``space'' functor to $X_{\bullet}$ and $Y_{\bullet}$, respectively. 
The fact that the $m$-th pointed simplicial set functor preserves limits and colimits now implies that the map $(\tilde{L}_n(Y_{\bullet}))(m) \rightarrow Y_n(m)$
agrees, up to an isomorphism in the domain, with the map
$\nu\co \tilde{L}_n\left(Y_{\bullet}(m)\right) \rightarrow Y_n(m)$.
Since $(-)_\ell$ preserves finite limits and finite colimits, the induced map $(\nu)_\ell$ agrees,
up to an isomorphism in the domain, with the map
$\nu_Y\co \tilde{L}_n\left((Y_{\bullet})_\ell(m)\right) \rightarrow (Y_n)_\ell(m)$.
Similarly, applying $(-)_\ell$ to the map $\tilde{L}_n\left(X_{\bullet}(m)\right) \rightarrow X_n(m)$ yields, up to an isomorphism in the domain,
the map 
$\nu_X\co \tilde{L}_n\left((X_{\bullet})_\ell(m)\right) \rightarrow (X_n)_\ell(m)$.

Here is the relevant consequence: 
we can choose an integer $k\in \{ 0, \dots , n-1\}$ such that 
\[ x_1\in ((\tilde{L}_nY_{\bullet})(m))_\ell \cong \tilde{L}_n((Y_{\bullet})(m))_\ell \] is in the $k$th coproduct summand
in $\tilde{L}_n(Y_{\bullet}(m))_\ell = \left(\coprod_{i=0}^{n-1} (Y_{n-1}(m))_\ell\right)/\sim$. 
Then $\overline{d}_k(\nu_Y(x_1)) = x_1$ by design.
Now the element $\overline{d}_k(x_0)\in \tilde{L}_n((X_{\bullet})_\ell(m))$
has two important features: we have equalities
\begin{align*} 
 (\tilde{L}_n((f_{\bullet})_\ell(m)))(\overline{d}_k(x_0)) 
  &= \left(\overline{d}_k((f_n)_\ell(m))\right)(x_0) \\
  &= \overline{d}_k(\nu_Y(x_1)) \\
  &= x_1,
\end{align*}
and we have equalities
\begin{align*} 
 \left(((f_n)(m))_\ell\right)\left( \nu_X(\overline{d}_k(x_0))\right) 
  &= \left( \nu_Y( \tilde{L}_n((f_{\bullet})_\ell)(m))\right)(\overline{d}_k(x_0)) \\
  &= \nu_Y(x_1) \\
  &= ((f_n)(m))_\ell(x_0).
\end{align*}
Since each morphism of symmetric spectra $f_n\co X_n \rightarrow Y_n$ is a flat cofibration, it is also a levelwise cofibration (see Corollary~3.12 in Schwede's book~\cite{schwedebook}), 
hence each $f_n(m)$ is a cofibration in $\sSet_*$ and hence each
$((f_n)(m))_s$ is a monomorphism of pointed sets, hence left-cancellable, and 
so \[ \nu_X(\overline{d}_k(x_0)) = x_0.\]

Now $\nu_X(\overline{d}_k(x_0)) = x_0$ and 
$(\tilde{L}_n((f_{\bullet})_\ell(m)))(\overline{d}_k(x_0))  = x_1$ together imply
that the element $\overline{d}_k(x_0) \in \tilde{L}_n((X_{\bullet})_\ell)(m)$ maps to 
$x_0$ and to $x_1$ under the maps in the diagram
\[\xymatrix{
 \tilde{L}_n((X_{\bullet})_\ell(m)) \ar[rr]^{\tilde{L}_n((f_{\bullet})_\ell(m))} \ar[d]_{\nu_X} & & \tilde{L}_n((Y_{\bullet})_\ell(m)) \ar[d]_{\nu_Y} \\ 
 ((X_n)(m))_\ell \ar[rr]^{((f_n)(m))_\ell} & & (Y_n(m))_\ell}.\]

Consequently, $x_0$ and $x_1$ represent the same element in the pushout $(\tilde{L}_nY_{\bullet}(m))_\ell\coprod_{(\tilde{L}_nX_{\bullet}(m))_\ell} (X_n(m))_\ell.$
Consequently, the map 
\[ \tilde{L}_n(((Y_{\bullet})(m))_\ell)\coprod_{\tilde{L}_n(((X_{\bullet})(m))_\ell)} (X_n(m))_\ell  \rightarrow (Y_n(m))_\ell\]
given by the universal property of the pushout
is injective. Since the evaluation functor $(-)_\ell$ commutes with finite limits and colimits, hence also with pushouts and with the formation of latching objects,
we get that 
the map given by the universal property of the pushout
\[ \left( \tilde{L}_n(Y_{\bullet})(m) \coprod_{\tilde{L}_n(X_{\bullet})(m)} X_n(m)\right)_\ell  \rightarrow (Y_n(m))_\ell\]
is also a monomorphism, hence that 
\[  \tilde{L}_n(Y_{\bullet})(m) \coprod_{\tilde{L}_n(X_{\bullet})(m)} X_n(m)  \rightarrow Y_n(m)\]
is a cofibration in $sSet_*$ for each $n$ and $m$.
Hence
\[  \tilde{L}_n(Y_{\bullet}) \coprod_{\tilde{L}_n(X_{\bullet})} X_n  \rightarrow Y_n\]
is a levelwise cofibration in $\Sp_{\sSet_*}$, hence $f_{\bullet}$ is a Reedy levelwise cofibration, as claimed.

If we furthermore assume that $f_{\bullet}$ is a pointwise positive flat cofibration, then $f_n(0)\co X_n(0) \rightarrow Y_n(0)$ and $L_nf(0)\co L_nX(0) \rightarrow L_nY(0)$ are isomorphisms for all $n$.
Consequently, the map 
 \[ \tilde{L}_n(Y_{\bullet})(0) \coprod_{\tilde{L}_n(X_{\bullet})(0)} X_n(0) \rightarrow Y_n(0)\]
is an isomorphism,
and consequently the canonical comparison map
\[ \left(\tilde{L}_n(Y_{\bullet}) \coprod_{\tilde{L}_n(X_{\bullet})} X_n\right)(0) \rightarrow Y_n(0)\]
is an isomorphism, which makes $f_{\bullet}$ not only a Reedy levelwise cofibration but a Reedy {\em positive} levelwise cofibration.
\end{case}
\end{proof}

\begin{theorem}\label{reedy flat cofibration thm}
Let $X_{\bullet}\stackrel{f_{\bullet}}{\longrightarrow} Y_{\bullet}$ be a morphism of simplicial symmetric spectra in $\sSet_*$. Make the following assumptions:
\begin{itemize}
\item The map $f_{\bullet}$ is a pointwise flat cofibration. That is, for each nonnegative integer $n$, the map $X_n \stackrel{f_n}{\longrightarrow} Y_n$ is a flat cofibration.
\item Both $X_{\bullet}$ and $Y_{\bullet}$ are Reedy flat-cofibrant.
\end{itemize}

Then $f_{\bullet}$ is a Reedy flat cofibration, and consequently the map of geometric realizations $\left| f_{\bullet}\right| \co \left| X_{\bullet}\right|\rightarrow  \left|Y_{\bullet} \right|$ is a flat cofibration.

If we furthermore assume that $f_{\bullet}$ is a pointwise {\em positive} flat cofibration, then $f_{\bullet}$ is a Reedy positive flat cofibration, hence $\left| f_{\bullet}\right|$ is a positive flat cofibration. 
\end{theorem}
\setcounter{case}{0}
\begin{proof}
Let $\underline{PO}^n$ denote the symmetric spectrum in $\sSet_*$ defined to be the pushout in the square
\[ \xymatrix{
 \tilde{L}_n(X_{\bullet}) \ar[r]^{\tilde{L}_n(f_{\bullet})} \ar[d]^{\tilde{\nu}_n} & \tilde{L}_n(Y_{\bullet}) \ar[d] \\
 X_n \ar[r] & \underline{PO}^n, }\]
and let $PO(n,m)$ denote the pointed simplicial set defined to be the pushout in the square
\[ \xymatrix{
 L_m(\underline{PO}^n) \ar[d]^{\nu_n} \ar[r] & L_m(Y_n) \ar[d] \\
  \underline{PO}^n(m) \ar[r] & PO(n,m).}\]
We need to show that the canonical map $c_{n,m}\co PO(n,m) \rightarrow Y_n(m)$, given by the universal property of the pushout, is a cofibration in $\sSet_*$ for all nonnegative integers $m$ and $n$.
Indeed, fix $n$, and suppose we have shown that $c_{n,m}$ is a cofibration for all values of $m$. This is exactly the condition required for the canonical map $\underline{PO}^n \rightarrow Y_n$, given by the universal property of the pushout, to be a flat cofibration. If we show that this canonical map is a flat cofibration for all $n$, then we have shown $f_{\bullet}$ a Reedy flat cofibration, by definition.

To show that $c_{n,m}$ is a cofibration, we explicitly check for each nonnegative integer $\ell$ that, after applying the functor $(-)_\ell$ as in Theorem \ref{reedy levelwise cofibrancy lemma}, the map $(c_{n,m})_\ell$ is a monomorphism. {\em Throughout, we freely make use of the fact that $(-)_\ell$ preserves finite limits and finite colimits, and consequently sends latching objects to latching objects,} as discussed in the proof of Theorem~\ref{reedy levelwise cofibrancy lemma}.
Suppose that $x,y\in (PO(n,m))_\ell$ are elements satisfying $(c_{n,m})_\ell(x) = (c_{n,m})_\ell(y)$. Then, since pushouts in simplicial sets are computed pointwise, 
there are three possibilities:

\begin{case}
Suppose $x,y$ are the images of elements $\overline{x},\overline{y} \in (L_m(Y_n))_\ell$ under the map $(L_m(Y_n))_\ell \rightarrow (PO(n,m))_\ell$. 
$Y_{\bullet}$ is assumed to be Reedy flat cofibrant, hence is levelwise flat-cofibrant, hence $Y_n$ is flat-cofibrant and hence the map $(L_m(Y_n))_\ell \rightarrow (Y_n(m))_\ell$ is a monomorphism of pointed sets for all $m$, and consequently $x=y$.
\end{case}
\begin{case} 
Suppose $y$ is the image of some element $\overline{y}$ under the map $(L_m(Y_n))_\ell \rightarrow (PO(n,m))_\ell$ and $x$
is the image of some element $\overline{x}$ under the map $(\underline{PO}^n(m))_\ell \rightarrow (PO(n,m))_\ell$.  
Then there are two sub-cases to consider:
\begin{subcase}
Suppose $x$ is the image of some element $\overline{x}^{\prime}\in (X_n(m))_\ell$ under the map $(X_n(m))_\ell \rightarrow (\underline{PO}^n(m))_\ell$.
Since $f_{\bullet}$ is assumed a levelwise flat cofibration, the map $X_n \stackrel{f_n}{\longrightarrow} Y_n$ is a flat cofibration,
and hence the map \[ (X_n(m))_\ell \coprod_{(L_m(X_n))_\ell} (L_m(Y_n))_\ell \rightarrow (Y_n(m))_\ell\] is a monomorphism of pointed sets. This map factors through a map 
\[ (X_n(m))_\ell \coprod_{(L_m(X_n))_\ell} (L_m(Y_n))_\ell \lra (PO(n,m))_\ell \] 
which is also a monomorphism since it is the first map in a composite map that is a monomorphism. By commutativity of the relevant diagrams, this implies that $x=y$. 
\end{subcase}
\begin{subcase}
Suppose $x$ is the image of some element $\overline{x}^{\prime}\in ((\tilde{L}_n(Y_{\bullet}))(m))_\ell$ under the map\newline $((\tilde{L}_n(Y_{\bullet}))(m))_\ell \rightarrow (\underline{PO}^n(m))_\ell$. 
Then $\overline{x}^{\prime},\overline{y}^{\prime}$ each define an element $x^{\prime},y^{\prime}$ in the pushout set 
\[ \left( (\tilde{L}_n(Y_{\bullet}))(m) \coprod_{L_m(\tilde{L}_n(Y_{\bullet}))} L_m(Y_n)\right)_\ell,\]
and these two elements map to the same element of $(Y_n(m))_\ell$, since
$(c_{n,m})_\ell(x) = (c_{n,m})_\ell(y)$. Since $Y_{\bullet}$ is Reedy flat-cofibrant, the map of pointed sets
\[ \left((\tilde{L}_n(Y_{\bullet}))(m) \coprod_{L_m(\tilde{L}_n(Y_{\bullet}))} L_m(Y_n)\right)_\ell \rightarrow (Y_n(m))_\ell\]
is a monomorphism; consequently $x^{\prime} = y^{\prime}$, and hence
$x^{\prime},y^{\prime}$ both pull back to a single element $u\in (L_m(\tilde{L}_n(Y_{\bullet})))_\ell$, and the image of this element under the map
\[ (L_m(\tilde{L}_n(Y_{\bullet})))_\ell \rightarrow (PO(n,m))_\ell\]
is equal to both $x$ and $y$; hence $x=y$.
\end{subcase}
\end{case}
\begin{case}
Suppose $x,y$ are the images of elements $\overline{x},\overline{y}\in (\underline{PO}^n(m))_\ell$ under the map \[ (\underline{PO}^n(m))_\ell \rightarrow (PO(n,m))_\ell. \] 
By Theorem~\ref{reedy levelwise cofibrancy lemma} the map $\underline{PO}^n\rightarrow  Y_n$ is a levelwise cofibration so $((\underline{PO}^n)_m)_\ell\rightarrow (Y_n(m))_\ell$ is a monomorphism and therefore $\bar{x}=\bar{y}$, which implies $x=y$. 
\end{case}

The above argument shows that the canonical comparison map
\[ c_{n,m}\co  PO(n,m)=  L_m(Y_n)\coprod_{L_m(\underline{PO}^n)} \underline{PO}^n(m) \rightarrow Y_n(m)\]
is a cofibration in $\sSet_*$ for all $m$ and $n$, hence that the canonical map of symmetric spectra in pointed simplicial sets
\begin{equation}\label{comparison map 40834} \tilde{L}_n(Y_{\bullet})\coprod_{\tilde{L}_n(X_{\bullet})} X_n = \underline{PO}^n \rightarrow Y_n\end{equation}
 is a flat cofibration for all $n$, hence that
$f_{\bullet}\co X_{\bullet} \rightarrow Y_{\bullet}$ is a Reedy flat cofibration.
If we furthermore assume that $f_{\bullet}$ is a pointwise {\em positive} flat cofibration, then by Theorem~\ref{reedy levelwise cofibrancy lemma}, the map~\eqref{comparison map 40834} is a positive levelwise cofibration in addition to being a flat cofibration; so $f_{\bullet}$ is a Reedy positive flat cofibration, as claimed.
\end{proof}

\section{Application}  
We now give an example of a situation where the main theorem is useful in a practical situation. Suppose we have an explicit model for the pointwise cofiber of a map of simplicial symmetric spectra in pointed simplicial sets $X_{\bullet}\rightarrow Y_{\bullet}$, where by pointwise cofiber we mean a simplicial object $Z_{\bullet}$ with a map $Y_{\bullet}\rightarrow Z_{\bullet}$ so that, for each $n$, the object $Z_n$ is the colimit (the categorical colimit, not a homotopy colimit requiring factorizations or replacements of the maps in the diagram!) of the diagram
\[ 
	\xymatrix{ 
		X_n\ar[r] \ar[d] &  Y_n  \\
		0 
	} 
\]
If the map of simplicial symmetric spectra in pointed simplicial sets $X_{\bullet}\rightarrow Y_{\bullet}$ is a pointwise flat cofibration then $Z_n$ is isomorphic to the  the homotopy cofiber of the map $X_n\rightarrow Y_n$. 
If, in addition, the objects $X_{\bullet}$ and  $Y_{\bullet}$ are Reedy cofibrant simplicial objects in $\Sp_{\sSet_*}$,  then Theorem \ref{reedy flat cofibration thm} applies, elevating the map $X_{\bullet}\rightarrow Y_{\bullet}$ to a Reedy cofibration so that 
\begin{equation}\label{cof geom real} |X_{\bullet}|\rightarrow |Y_{\bullet}|\rightarrow |Z_{\bullet}|. \end{equation}
is a homotopy cofiber sequence in spectra.

In other words, we do not need to cofibrantly replace any of the objects or use any factorization systems and yet the maps $|X_{\bullet}|\rightarrow |Y_{\bullet}|$ and $0\rightarrow |Z_{\bullet}|$ are still flat cofibrations so $|Z_{\bullet}|$ is indeed the homotopy cofiber of the map $|X_{\bullet}|\rightarrow |Y_{\bullet}|$. Consequently, $|Z_{\bullet}|$ is a spectrum constructed by an explicit point-set, levelwise pushout construction, but whose homotopy groups are concretely computable via the long exact sequence of the homotopy cofiber sequence \eqref{cof geom real}. 
We use this in a critical way in \cite{thhmay} for example. 

\bibliography{salch}{}
\bibliographystyle{plain}
\end{document}